\documentclass[reqno]{amsart}

\usepackage{amsmath,amsthm,amsfonts, eucal, amssymb}
\usepackage[latin1]{inputenc}
\usepackage[T1]{fontenc}
\usepackage[english]{babel}
\usepackage{hyperref}
\usepackage{graphicx, color}
\usepackage{epsf}
\usepackage{pst-all}

\newtheorem{theorem}{Theorem}

\newtheorem{lemma}{Lemma}

\theoremstyle{definition}
\newtheorem{definition}{Definition}

\theoremstyle{remark}
\newtheorem{remark}{Remark}

\begin{document}

\title{THE MINIMUM CROSSING NUMBER OF ESSENTIAL TANGLES}

\author{Jo\~{a}o Miguel Nogueira} 

\address{CMUC, Department of Mathematics, University of Coimbra, Apartado 3008, 3001-454 Coimbra, Portugal\\
nogueira@mat.uc.pt, ams@mat.uc.pt}

\author{Ant\'{o}nio Salgueiro}

\thanks{Both authors were partially supported by the \textit{Centro de Matem\'{a}tica da Universidade de Coimbra} (CMUC), funded by the European Regional Development Fund through the program COMPETE and by the Portuguese Government through the FCT - \textit{Funda\c{c}\~{a}o para a Ci\^{e}ncia e Tecnologia} under the project PEst-C/MAT/UI0324/2011.}

\begin{abstract}
In this paper we compute the sharp lower bounds for the crossing number of $n$-string $k$-loop essential tangles. For essential tangles with only string components, we characterise the ones with the minimum crossing number for a given number of components, both when the tangle has knotted strings or only unknotted strings.
\end{abstract}



\maketitle

\section{Introduction}

\indent Tangles were first introduced by Conway in \cite{Conway} where he defines and classifies $2$-string (\textit{rational}) tangles as an instrument to list knots. Since then, the concept of tangle has been important in knot theory and its applications, and $3$-manifold topology. For instance, in \cite{Kir-Lick} Kirby and Lickorish prove that any knot is concordant to a prime knot by introducing the concept of essential tangle. Other example is the work of Lickorish in \cite{Lickorish} and its extension by G\'omez-Larra\~naga in \cite{Larranaga} where conditions for knot or link primeness are given based on tangle decompositions of a knot or link. In $3$-manifold topology we can also find a pertinent use of tangles in the study of Dehn fillings and related problems as the cabling conjecture \cite{Men-Zhang}. We also have the presence of tangles in applied mathematics as in the study of the DNA topology. This application was pioneered by Ernst and Sumners in \cite{Ernst-Sumners}, and in the survey paper \cite{Buck} by Buck we can find a concise explanation on how tangles are used to study the DNA topology.\\     

There have been several other studies on tangles and their diagrams. For instance, in the work of Conway in \cite{Conway}, besides introducing the concept of tangle, the rational tangles with two strings are also classified. Also, in \cite{Kanenobu} Kanenobu, Saito and Satoh classify up to isotopy $2$-string $k$-loop prime tangles with up to seven crossings. In a more general approach, the paper \cite{Bogdanov} presents a study on the enumeration of $k$-string tangle projections, and a computerized enumeration of alternating tangles is given in \cite{Jacobsen}. Also, in \cite{Cochran} Cochran and Ruberman present invariants of $2$-string tangles and use them to distinguish some knots from their mutants.\\

In this paper we compute the sharp lower bound for the crossing number of all $n$-string $k$-loop essential tangles, for each $n\geq 1$ and $k\geq 0$. In the case when the tangles have no loops, only strings, we also characterize which $n$-string essential tangles have the minimum crossing number possible. This question was initially motivated by an observation of Buck in \cite{Buck} that $2$-string essential tangles have crossing number at least five and citing the work of Bleiler in \cite{Bleiler} for this statement. The aim of this paper is to present a sharp lower bound on the crossing number of $n$-string $k$-loop essential tangles. The results obtained are presented in the following theorems.

\begin{theorem}\label{tangle general}
Let $\mathcal{T}$ be a $n$-string essential tangle. Then $$c(\mathcal{T})\geq 2n+1.$$

We have $c(\mathcal{T})=2n+1$ if and only if $\mathcal{T}$ is equivalent to the tangle in Fig. \ref{figure:tangle general}.

\end{theorem}

\begin{figure}[ht]

\begin{center}

\psset{unit=6mm,algebraic=true}

\begin{pspicture}(-3,-3)(3,3)

\pscircle[linewidth=0.1pt](0,0){3}

\pscurve(-1.2,2.74955)(-1.2,1)(-1.5,0.2) \pscurve(-1.7,-0.1)(-1.8,-0.5)(-1.2,-1)(-0.8,-1.15) \pscurve(-0.4,-1.2)(0.1,-1.2)(0.5,-1.15)

\pscurve(0.75,-1.1)(0.9,-1.05)(1.1,-0.95) \pscurve(1.35,-0.8)(1.7,0)(1.35,0.8)(1.1,0.95)(0.9,1.05)(0.75,1.1)(0.5,1.15)(0.1,1.2)(-0.6,1.15)(-1.0,1.05) \pscurve(-1.4,0.9)(-1.8,0.5)(-1.3,-0.5)(-1.2,-0.8) \psline(-1.2,-1.2)(-1.2,-2.74955)

\psline(1.2,2.74955)(1.2,1.1) \psline(1.2,0.5)(1.2,-2.74955)

\psline(0.6,2.93939)(0.6,1.3) \psline(0.6,0.8)(0.6,-2.93939)

\psline(-0.6,2.93939)(-0.6,1.3) \psline(-0.6,0.8)(-0.6,-2.93939)

\rput(0,2){$\cdots$}

\rput(0,0){$\cdots$}

\rput(0,-2){$\cdots$}

\end{pspicture}

\end{center}

\caption{The $n$-string essential tangle with the minimum crossing number.}

\label{figure:tangle general}

\end{figure}
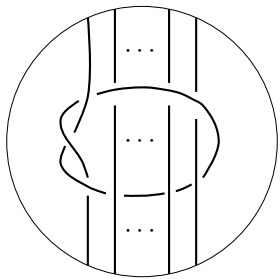


\begin{theorem}\label{tangle unknotted strings}
Let $\mathcal{T}$ be a $n$-string essential tangle, with all strings unknotted. Then $$c(\mathcal{T})\geq 2n+2.$$

We have $c(\mathcal{T})=2n+2$ if and only if $T$ is equivalent to one of the tangles in Fig. \ref{figure:tangle unknotted strings}.

\end{theorem}

\begin{figure}[ht]

\begin{center}

\psset{unit=6mm,algebraic=true}

$\begin{array}{ccc}

\begin{pspicture}(-2.2,-4)(4,4)

\pscircle[linewidth=0.1pt](0,0){3}

\psline(-1.8,2.4)(-1.8,1.7) \psline(-1.8,1.3)(-1.8,-0.6) \psline (-1.8,-1.1)(-1.8,-2.4)
\pscurve(-1.4,-1)(-1.8,-0.8)(-2.35,0)(-2.4,0.2)
\pscurve(-2.4,0.6)(-1.8,1.5)(-1.3,1.4)(-1.2,1.2)
\pscurve(-0.7,-1.17)(-0.9,-1.15)(-1.1,-1.1) \pscurve(-0.4,-1.2)(0.1,-1.2)(0.5,-1.15)

\pscurve(0.75,-1.1)(0.9,-1.05)(1.1,-0.95) \pscurve(1.35,-0.8)(1.7,0)(1.35,0.8)(1.1,0.95)(0.9,1.05)(0.75,1.1)(0.5,1.15)(0.1,1.2)(-0.6,1.15)(-1.6,0.9)
\psline (-2,0.7)(-3,0) 
\psline(-1.2,0.8)(-1.2,-2.74955)

\psline(1.2,2.74955)(1.2,1.1) \psline(1.2,0.5)(1.2,-2.74955)

\psline(0.6,2.93939)(0.6,1.3) \psline(0.6,0.8)(0.6,-2.93939)

\psline(-0.6,2.93939)(-0.6,1.3) \psline(-0.6,0.6)(-0.6,-2.93939)

\rput(0,2){$\cdots$}

\rput(0,0){$\cdots$}

\rput(0,-2){$\cdots$}

\rput(3.5;270){(a)}
\end{pspicture}&

\begin{pspicture}(-4,-4)(4,3)

\pscircle[linewidth=0.1pt](0,0){3}

\psline(1.2,-2.74955)(1.2,-0.8)
\pscurve(1.2,-0.8)(1.25,-0.5)(1.4,-0.2)
\pscurve(1.6,0.2)(1.8,0.8)(1.2,1.2)(0,1.4)(-1.2,1.2)(-1.8,0.8)(-1.6,0.2)
\pscurve(-1.4,-0.2)(-1.25,-0.5)(-1.2,-0.8)
\psline(-1.2,-0.8)(-1.2,-2.74955)

\psline(0.6,-2.93939)(0.6,1.2)
\psline(0.6,1.8)(0.6,2.93939)
\psline(-0.6,-2.93939)(-0.6,1.2)
\psline(-0.6,1.8)(-0.6,2.93939)
\rput(0,2){$\cdots$}
\rput(0,0){$\cdots$}
\rput(0,-2){$\cdots$}

\psline(1.2,2.74955)(1.2,1.5)
\pscurve(1.2,1)(1.3,0.4)(1.8,-0.8)(1.3,-1.1)
\psline(-1.2,2.74955)(-1.2,1.5)
\pscurve(-1.2,1)(-1.3,0.4)(-1.8,-0.8)(-1.3,-1.1)
\pscurve(1.05,-1.15)(0.9,-1.2)(0.75,-1.2)
\pscurve(0.45,-1.25)(0,-1.3)(-0.45,-1.25)
\pscurve(-1.05,-1.15)(-0.9,-1.2)(-0.75,-1.2)

\rput(3.5;270){(b)}
\end{pspicture}&

\begin{pspicture}(-4,-4)(4,3)

\pscircle[linewidth=0.1pt](0,0){3}

\psline(1.2,-2.74955)(1.2,-1.5)
\pscurve(1.2,-1)(1.3,-0.4)(1.8,0.8)(1.3,1.2)
\pscurve(1,1.3)(0,1.4)(-1.2,1.2)(-1.8,0.8)(-1.6,0.2)
\pscurve(-1.4,-0.2)(-1.25,-0.5)(-1.2,-0.8)
\psline(-1.2,-0.8)(-1.2,-2.74955)

\psline(0.6,-2.93939)(0.6,1.2)
\psline(0.6,1.8)(0.6,2.93939)
\psline(-0.6,-2.93939)(-0.6,1.2)
\psline(-0.6,1.8)(-0.6,2.93939)
\rput(0,2){$\cdots$}
\rput(0,0){$\cdots$}
\rput(0,-2){$\cdots$}

\psline(1.2,2.74955)(1.2,1)
\pscurve(1.2,1)(1.3,0.4)(1.4,0.2)
\pscurve(1.6,-0.2)(1.8,-0.8)(1.3,-1.1)(0.75,-1.2)
\psline(-1.2,2.74955)(-1.2,1.5)
\pscurve(-1.2,1)(-1.3,0.4)(-1.8,-0.8)(-1.3,-1.1)
\pscurve(0.45,-1.25)(0,-1.3)(-0.45,-1.25)
\pscurve(-1.05,-1.15)(-0.9,-1.2)(-0.75,-1.2)

\rput(3.5;270){(c)}
\end{pspicture}
\end{array}$

\end{center}

\caption{The all strings unknotted $n$-string essential tangles with the minimum crossing number.}

\label{figure:tangle unknotted strings}

\end{figure}
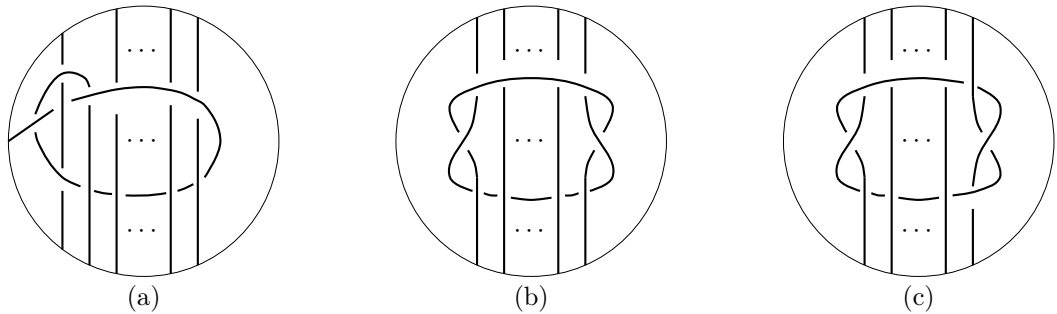


\begin{theorem}\label{tangle closed strings}
Let $\mathcal{T}$ be a $n$-string $k$-loop essential tangle, with $k\geq 1$. Then $$c(\mathcal{T})\geq 2(n+k)-2.$$

If $k=1$, then $c(\mathcal{T})=2n$ if and only if $\mathcal{T}$ is equivalent to the tangle in Fig. \ref{figure:tangle closed strings}.

\end{theorem}

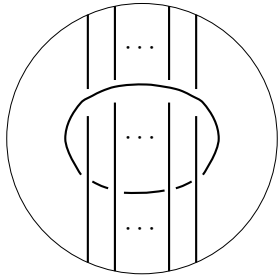
\begin{figure}[ht]

\begin{center}
\psset{unit=6mm,algebraic=true}
\begin{pspicture}(-3,-3)(3,3)
\pscircle[linewidth=0.1pt](0,0){3}
\pscurve(-0.4,-1.2)(0.1,-1.2)(0.5,-1.15) \pscurve(0.75,-1.1)(0.9,-1.05)(1.1,-0.95)
\pscurve(1.35,-0.8)(1.7,0)(1.35,0.8)(1.1,0.95)(0.9,1.05)(0.75,1.1)(0.5,1.15)(0.1,1.2)(-0.6,1.15)(-0.9,1.05)(-1.1,0.95)(-1.35,0.8)(-1.7,0)(-1.35,-0.8)
\pscurve(-1.1,-0.95)(-0.9,-1.05)(-0.75,-1.1)
\psline(1.2,2.74955)(1.2,1.1) \psline(1.2,0.5)(1.2,-2.74955)
\psline(0.6,2.93939)(0.6,1.3) \psline(0.6,0.8)(0.6,-2.93939)
\psline(-0.6,2.93939)(-0.6,1.3) \psline(-0.6,0.8)(-0.6,-2.93939)
\psline(-1.2,2.74955)(-1.2,1.1) \psline(-1.2,0.5)(-1.2,-2.74955)
\rput(0,2){$\cdots$}
\rput(0,0){$\cdots$}
\rput(0,-2){$\cdots$}

\end{pspicture}
\end{center}
\caption{The $n$-string $1$-loop essential tangle with the minimum crossing number.}
\label{figure:tangle closed strings}
\end{figure}

For a brief introduction to the concept of tangle, as in Conway's work, we refer to the book \cite{Adams} by Colin Adams, and throughout this paper we work in the piecewise linear category. In Section $2$ we introduce notation and prove some lemmas that define a base for the proofs of the main results in this paper. In sections $3$, $4$ and $5$, we prove Theorems \ref{tangle general}, \ref{tangle unknotted strings} and \ref{tangle closed strings}, respectively, using combinatorics in the diagram projection of tangles.

\section{Preliminaries}\label{Preliminares} 
A {\em $n$-string $k$-loop tangle} $\mathcal{T}=(B,\sigma)$ is a $3$-ball $B$ together with a compact $1$-sub\-ma\-ni\-fold $\sigma$, with boundary, that is a disjoint collection of $n\geq 1$ arcs (called {\em strings}) and $k\geq 0$ simple closed curves (called {\em loops}). When $k=0$ we say simply that $\mathcal{T}$ is a {\em$n$-string tangle}. In this paper we say that two tangles $\mathcal{T}_1$, $\mathcal{T}_2$ are {\em equivalent} if there is a homeomorphism of $(B,\partial B)$ sending $\mathcal{T}_1$ to $\mathcal{T}_2$.  A tangle $\mathcal{T}$ is said to be {\em essential} if $\sigma$ is a single knotted arc\footnote{An arc of $\sigma$ is {\em unknotted} in $B$ if it cobounds a disk embedded in $B$ together with an arc in $\partial B$, otherwise it is said to be {\em knotted}.}\; in $B$, or if $\sigma$ has more than one component and there is no properly embedded disk in $B$ separating the components of $\sigma$. Otherwise, we say that the tangle is {\em inessential}. All tangles discussed henceforth  (unless specified otherwise) are assumed to be essential tangles.

The {\em crossing number} of a tangle $\mathcal{T}$, denoted by $c(\mathcal{T})$, is the minimum crossing number of the diagrams of $\mathcal{T}$. The classification of the essential $n$-string $k$-loop tangles with the minimum crossing number up to isotopy follows from the classification up to homeomorphism described in this paper, by considering the mirror images of the tangles depicted in Figures 1, 2 and 3.\\

A {\em projection} of a tangle $\mathcal{T}$ is the image $p(\mathcal{T})$ of the tangle by an orthogonal projection $p$ to a plane such that $p(B)$ is a disk, the preimage of each point of $p(\sigma)$ has at most two points, and there is a finite number of double points, which are called the {\em crossings} of the projection. A projection always exists in the piecewise linear category. The connected components of the complement of the crossings are called the {\em segments} of the projection. If the crossings are decorated with broken lines to show the overcrosses and undercrosses, then we get a {\em diagram} of $\mathcal{T}$. The {\em crossing number} of a projection or a diagram of $\mathcal{T}$ is its number of crossings. A {\em minimal diagram} of $\mathcal{T}$ is a diagram with the minimum crossing number $c(\mathcal{T})$.

\begin{definition}
There are four segments adjacent to each crossing. Two of these segments are {\em opposite} if their preimages by the projection are adjacent, and {\em consecutive} otherwise. A crossing is called {\em outermost} if it is the first crossing of a string for some orientation, and {\em inner} otherwise.
\end{definition}

	
\begin{lemma}
In a diagram of an essential tangle $\mathcal{T}$, each string has at least two crossings with the other strings.
\end{lemma}
\begin{proof}
Suppose that a string $s$ has at most one crossing with the other strings. Then there is a disk $D$ with boundary $\partial D=s\cup \alpha$, where $\alpha$ is a curve in $\partial B$. A small neighborhood of $D$ is then a ball that contains only the string $s$, which shows that $\mathcal{T}$ is not essential.
\end{proof}

The next lemma is a base for the theorems in this paper and provides a common weaker version of all of them.
\begin{lemma}
\label{injection} If $\mathcal{T}$ is a $n$-string $k$-loop essential tangle, then $c(\mathcal{T})\geq 2n$.

\end{lemma}

\begin{proof}
Consider a minimal diagram of $\mathcal{T}$. Since each string has at least two crossings with the other strings, we can associate to each end of a string an outermost crossing. If two ends were associated with the same crossing, as in Fig. \ref{figure:non injective}(a), then we could isotope the two ends through the region between them, as in Fig. \ref{figure:non injective}(b), reducing the number of crossings by one. This contradicts the diagram being minimal, and therefore this shows that the association is injective. Hence $c(\mathcal{T})\geq 2n$.
\end{proof}

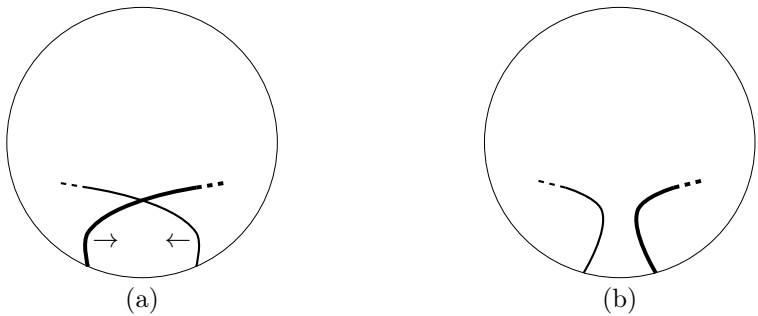
\begin{figure}[ht]

\begin{center}

\psset{unit=6mm,algebraic=true}

$\begin{array}{cc}

\begin{pspicture}(-5,-4)(5,3)

\pscircle[linewidth=0.1pt](0,0){3}

\pscurve(1.2,-2.74955)(1.2,-2)(-1.2,-1)

\psline[linestyle=dashed,dash=2pt 2pt](-1.2,-1)(-1.8,-0.9)

\pscurve[linewidth=1.5pt](-1.2,-2.74955)(-1.2,-2)(1.2,-1)

\psline[linewidth=1.5pt,linestyle=dashed,dash=2pt 2pt](1.2,-1)(1.8,-0.9)

\rput(0.8,-2.2){$\leftarrow$}

\rput(-0.8,-2.2){$\rightarrow$}

\rput(3.5;270){(a)}

\end{pspicture}&

\begin{pspicture}(-5,-4)(5,3)

\pscircle[linewidth=0.1pt](0,0){3}

\pscurve(-0.8,-2.89137)(-0.4,-1.5)(-1.2,-1)

\psline[linestyle=dashed,dash=2pt 2pt](-1.2,-1)(-1.8,-0.85)

\pscurve[linewidth=1.5pt](0.8,-2.89137)(0.4,-1.5)(1.2,-1)

\psline[linewidth=1.5pt,linestyle=dashed,dash=2pt 2pt](1.2,-1)(1.8,-0.85)

\rput(3.5;270){(b)}
\end{pspicture}

\end{array}$

\end{center}

\caption{In a minimal diagram, two strings do not share an outermost crossing.}

\label{figure:non injective}

\end{figure}

\begin{definition}
In a minimal diagram of a tangle, each outermost crossing $c$ defines a single segment $s$ containing an end of a string. Then $s$ is referred to as the {\em external segment} of $c$. The opposite segment with respect to $c$ is called the {\em internal segment} and the other two are the {\em lateral segments} of $c$.
\end{definition}

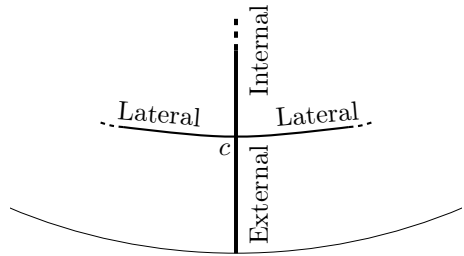
\begin{figure}[ht]

\begin{center}

\psset{unit=6mm,algebraic=true}

\begin{pspicture*}(-5,-3)(5,2.5)

\pscircle[linewidth=0.1pt](0,10){13}

\pscurve(-2.5,-0.2)(-0.5,-0.4)(0.5,-0.4)(2.5,-0.2)

\pscurve[linestyle=dashed,dash=2pt 2pt](-2.5,-0.2)(-2.75,-0.17)(-3,-0.1)
\pscurve[linestyle=dashed,dash=2pt 2pt](2.5,-0.2)(2.75,-0.17)(3,-0.1)

\psline[linewidth=1.5pt](0,-3)(0,1.5)
\psline[linewidth=1.5pt,linestyle=dashed,dash=2pt 2pt](0,1.5)(0,2.2)

\rput(-0.25,-0.7){$c$}
\rput(0.5,-1.7){\rotatebox{90}{External}}
\rput(0.5,1.5){\rotatebox{90}{Internal}}
\rput(1.8,0.1){\rotatebox{5}{Lateral}}
\rput(-1.7,0.1){\rotatebox{-5}{Lateral}}
\end{pspicture*}

\end{center}

\caption{External, internal and lateral segments of an outermost crossing.}

\label{figure:external, internal and lateral}

\end{figure}

\begin{lemma}\label{north west}
In a minimal diagram of a tangle, there is no segment which is both internal and lateral.

\end{lemma}

\begin{proof}
Suppose there is a minimal diagram of a tangle with a segment $s$ which is both internal and lateral. Hence, as illustrated in Fig. \ref{figure:second crossing}, we can proceed with an isotopy of $s$ and reduce the number of crossings in $\mathcal{D}$, which contradicts its minimality.\end{proof}

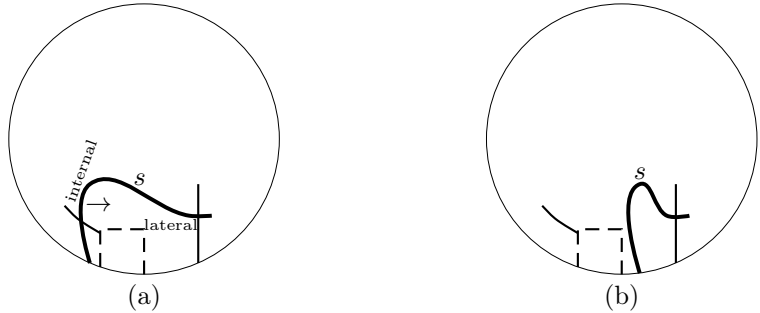
\begin{figure}[ht]

\begin{center}

$\begin{array}{cc}

\psset{unit=6mm,algebraic=true}

\begin{pspicture}(-5,-4)(5,4)

\pscircle[linewidth=0.1pt](0,0){3}

\pscurve[linewidth=1.5pt](-1.2,-2.74955)(-1.2,-1)(1,-1.7)(1.5,-1.7)

\pscurve(2.3;220)(2.3;230)(2.3;245)

\psline(1.2,-2.74955)(1.2,-1)

\psline[linestyle=dashed](-0.972022,-2.83816)(-0.972022,-2)(0,-2)(0,-3)

\rput(-1,-1.5){$\rightarrow$}

\rput(-0.1,-0.9){$s$}

\rput(-1.4,-0.7){\tiny \rotatebox{70}{internal}}
\rput(0.6,-1.9){\tiny lateral}
\rput(3.5;270){(a)}

\end{pspicture}&

\psset{unit=6mm,algebraic=true}

\begin{pspicture}(-5,-4)(5,4)

\pscircle[linewidth=0.1pt](0,0){3}

\pscurve[linewidth=1.5pt](0.4,-2.97321)(0.4,-1)(1,-1.7)(1.5,-1.7)

\pscurve(2.3;220)(2.3;230)(2.3;245)

\psline(1.2,-2.74955)(1.2,-1)

\psline[linestyle=dashed](-0.972022,-2.83816)(-0.972022,-2)(0,-2)(0,-3)

\rput(0.4,-0.75){$s$}

\rput(3.5;270){(b)}

\end{pspicture}

\end{array}$

\end{center}

\caption{In a minimal diagram, segments are not both internal and lateral.}

\label{figure:second crossing}

\end{figure}

\begin{remark}\label{no closed strings}
In a diagram of an $n$-string tangle there is no string containing only lateral segments, since all strings are arcs. Furthermore, there is at least one inner crossing, and at least two non opposite segments adjacent to the inner crossings are lateral.

\end{remark}

\section{Crossing number of essential tangles} In this section we study the minimum crossing number of $n$-string essential tangles and we prove Theorem \ref{tangle general}.\\

{\noindent \bf Proof of Theorem \ref{tangle general}.}
Let $\mathcal{T}$ be a $n$-string essential tangle and $\mathcal{D}$ a minimal diagram of $\mathcal{T}$. By the proof of Lemma \ref{injection}, $\mathcal{D}$ has $2n$ outermost crossings and by Remark \ref{no closed strings}, there is at least one inner crossing. Therefore $c(\mathcal{T})\geq 2n+1$.

Now suppose that $c(\mathcal{T})=2n+1$. Then there is exactly one inner crossing $c$. Of the four segments adjacent to $c$, two are lateral and the other two are internal. Again by Remark \ref{no closed strings}, the lateral segments cannot be opposite in $c$. Therefore the projection of $\mathcal{T}$ is as depicted in Fig. \ref{figure:general case}(a). In each {\em dotted region} there is a collection (possibly empty) of unknotted arcs that go around  the curve that contains the inner crossing (with one overcrossing and one undercrossing), as in Fig. \ref{figure:general case}(b). Each of these regions can slide along the curve and merge with the other regions, so that we could depict Fig. \ref{figure:general case}(a) with a single such region.

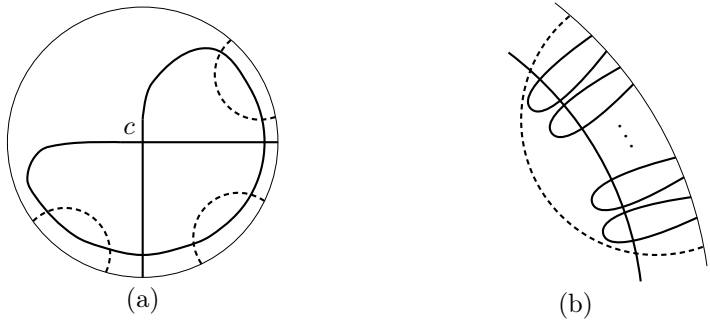
\begin{figure}[ht]

\begin{center}

\psset{unit=6mm,algebraic=true}

$\begin{array}{cc}

\begin{pspicture}(-5,-4)(3,4)

\pscircle[linewidth=0.1pt](0,0){3}

\psline(0,-3)(0,0.5)

\psline(-0.5,0)(3,0)

\pscurve(-0.5,0)(-1.8,-0.05)(-2.28,-0.2)(-2.55526, -0.820848)(-2.37846, -1.2)(-1.5, -2.07846)(-0.92117, -2.31822)(0,-2.5)(0.92117, -2.31822)(1.5, -2.07846)(2.37846, -1.2)(2.7, 0)(2.37846, 1.2)(1.5, 2.07846)(0.2,1.3)(0,0.5)

\psarc[linestyle=dashed,dash=2pt 1.5pt](3;-45){1}{55}{215}

\psarc[linestyle=dashed,dash=2pt 1.5pt](3;30){1}{130}{290}

\psarc[linestyle=dashed,dash=2pt 1.5pt](3;235){1}{-25}{135}

\rput(3.5;270){(a)}
\rput(-0.3,0.3){$c$}
\end{pspicture}&

\psset{unit=18mm,algebraic=true}

\begin{pspicture}(0,0)(3,0.3)

\psarc[linewidth=0.1pt](0,0){3}{8}{52}

\psarc[linestyle=dashed,dash=2pt 1.5pt](3;30){1}{130}{290}

\psarc(0,0){2.5}{7}{53}

\pscurve(3;15)(2.3;16.5)(3;18)

\pscurve(3;21)(2.3;22.5)(3;24)

\pscurve(3;42)(2.3;43.5)(3;45)

\pscurve(3;36)(2.3;37.5)(3;39)

\rput(2.75;30.5){\rotatebox{120}{$\cdots$}}

\rput(2,0.125){(b)}
\end{pspicture}

\end{array}$

\end{center}

\caption{(a) A projection of a $n$-string essential tangle with  $c(\mathcal{T})=2n+1$; (b) A projection of a dotted region.}

\label{figure:general case}

\end{figure}

If the string that contains the inner crossing is not alternating, then the tangle can be isotoped to a tangle with two less crossings. Therefore this string is knotted (by identifying its ends along $\partial B$ we obtain a trefoil) and all other strings are unknotted, and $\mathcal{T}$ is equivalent to the tangle represented in Fig. \ref{figure:tangle general}.

The tangle $\mathcal{T}$ is essential, since we can add a trivial tangle to $\mathcal{T}$ to obtain a link $L$ with $n$ components, one of which is a trefoil and the other $n-1$ components are trivial and parallel\footnote{Two strings of $\mathcal{T}$ are {\em parallel} if they cobound a disk together with two disjoint arcs in $\partial B$.}. With a suitable orientation, the linking number of $L$ is $n-1$. If $\mathcal{T}$ is not essential, then there is a disk in $B$ separating the trefoil string from the trivial strings. Then these strings could be isotoped outside of $B$ and $L$ would be splittable. Therefore the linking number of $L$ would be $0$, which is a contradiction to the linking number of $L$ being $n-1$.
{\hfill \footnotesize $\square$}

\section{Crossing number of essential tangles with all strings unknotted} This section is devoted to the proof of Theorem \ref{tangle unknotted strings}. We study the minimum crossing number of $n$-string essential tangles with all strings unknotted, and identify the respective tangles with this crossing number.

We start by showing the first part of the theorem. For this, let $\mathcal{T}$ be a $n$-string essential tangle with all strings unknotted. By Theorem \ref{tangle general}, $c(\mathcal{T})\geq 2n+1$ and, if $c(\mathcal{T})=2n+1$ then at least one string is knotted. Therefore $c(\mathcal{T})\geq 2n+2$.\

For the following lemmas we let $\mathcal{T}$ be a $n$-string essential tangle with all strings unknotted and $c(\mathcal{T})= 2n+2$, and $\mathcal{D}$ be a minimal diagram of $\mathcal{T}$. Then, from the proof of Lemma \ref{injection}, there are $2n$ outermost crossings and $2$ inner crossings in $\mathcal{D}$, that we denote by $c_a$ and $c_b$.

\begin{lemma}\label{two strings}
There are exactly two strings on the inner crossings of $\mathcal{D}$.

\end{lemma}

\begin{proof}
As $\mathcal{D}$ only has two inner crossings then there are at most eight segments adjacent to the inner crossings. From Remark \ref{no closed strings}, at least two segments adjacent to the inner crossings are lateral. Then, in $\mathcal{D}$ at most six segments adjacent to inner crossings can be internal. Hence, there are at most three strings on the inner crossings of $\mathcal{D}$.\\

Suppose there are three strings on the inner crossings. Then, from Lemma \ref{north west}, there are six internal segments adjacent to the inner crossings. From Remark \ref{no closed strings}, two segments adjacent to inner crossings are lateral. Therefore, as there are only two segments adjacent to inner crossings that are not internal, one component $s$ of $\sigma$ is on all outermost crossings containing all lateral segments but the two adjacent to the inner crossings. (See Fig. \ref{figure:three strings}.)

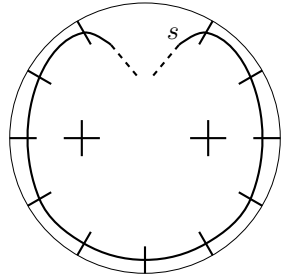
\begin{figure}[ht]

\begin{center}
\psset{unit=6mm,algebraic=true}
\begin{pspicture}(-3,-3)(3,4)
\pscircle[linewidth=0.1pt](0,0){3}
\psline(3;30)(2.4;30) \psline(3;60)(2.4;60) \psline(3;120)(2.4;120) \psline(3;150)(2.4;150) \psline(3;-30)(2.4;-30) \psline(3;-60)(2.4;-60) \psline(3;-90)(2.4;-90) \psline(3;-120)(2.4;-120) \psline(3;-150)(2.4;-150) \psline(3;0)(2.4;0) \psline(3;180)(2.4;180)
\pscurve(2.2;110)(2.7;120)(2.7;150)(2.7;-150)(2.7;-120)(2.7;-90)(2.7;-60)(2.7;-30)(2.7;30)(2.7;60)(2.2;70)
\psline[linestyle=dashed,dash=2pt 2pt](2.2;110)(1.4;97.5) \psline[linestyle=dashed,dash=2pt 2pt](2.2;70)(1.4;82.5)
\psline(-1.4,-0.4)(-1.4,0.4)
\psline(-1.8,0)(-1,0)
\psline(1.4,-0.4)(1.4,0.4)
\psline(1.8,0)(1,0)
\rput(2.4;75){$s$}
\end{pspicture}
\end{center}

\caption{If three strings are on the inner crossings, there is a string containing all lateral segments.}

\label{figure:three strings}

\end{figure}

\begin{figure}[ht]
\begin{center}
\psset{unit=6mm,algebraic=true}
$\begin{array}{cc}
\begin{pspicture}(-5,-4)(5,4)

\pscircle[linewidth=0.1pt](0,0){3}

\psline(3;60)(3;-60)
\psline(3;120)(3;-120)

\pscurve[linewidth=1.5pt](-3,0)(-1,0)(-1,0)(-0.4,0.2)(2.5;120)(2.5;180)(2.5;240)(2.5;270)(2.5;300)(2.5;0)(2.5;60)(0.4,0.2)(1,0)(1,0)(3,0)

\psarc[linestyle=dashed,dash=2pt 1.5pt](0,-3){1}{10}{170}
\rput(2.4;70){$s$}
\rput(3.5;270){(a)}
 
\end{pspicture}&
\begin{pspicture}(-5,-4)(5,4)

\pscircle[linewidth=0.1pt](0,0){3}

\psline(3;120)(3;-120)

\pscurve[linewidth=1.5pt](1.4,-2.6533)(1.4,0.4)(1.4,0.4)(1.2,1.4)(2.5;90)(2.5;120)(2.5;180)(2.5;240)(2.5;270)(2.5;300)(2.5;330)(2.3,-0.2)(1.8,0)(1.8,0)(1,0)(1,0)(0.7,-0.2)(0.55,-1)(0.52095,-2)(0.52095,-2)(0.52095,-2.95442)

\pscurve[linewidth=2pt](-3,0)(-1,0)(-1,0)(-0.7,-0.2)(-0.55,-1)(-0.52095,-2)(-0.52095,-2)(-0.52095,-2.95442)


\psarc[linestyle=dashed,dash=2pt 1.5pt](0,3){1}{-170}{-10}
\rput(2.2;50){$s$}
\rput(3.5;270){(b)}
\end{pspicture}
\end{array}$
\end{center}

\caption{If three strings are on the inner crossings, there is always an isotopy reducing the number of crossings.}
\label{figure:ab}
\end{figure}
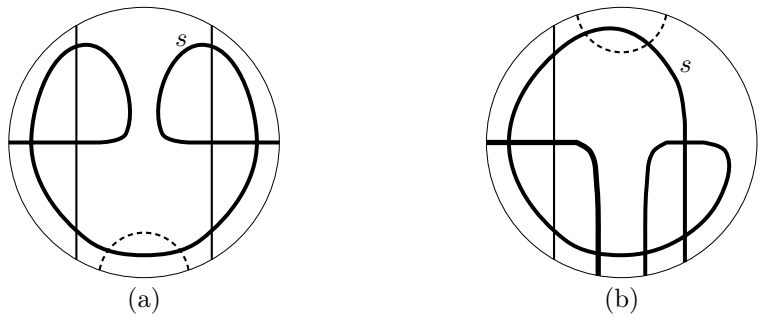

\noindent Either $s$ is on both inner crossings, as in Fig. \ref{figure:ab}(a), and each inner crossing is adjacent to three internal segments, or $s$ is on a single inner crossing and the other inner crossing has all ends adjacent to internal segments, as in Fig. \ref{figure:ab}(b). In the case of Fig. \ref{figure:ab}(a) we can reduce the number of crossings in the diagram by an isotopy of one of the strings on the inner crossings. And in the case of Fig. \ref{figure:ab}(b), as $s$ is unknotted, two of its three self-crossings in the diagram are not alternating, which allows us to isotope $\sigma$ reducing the crossing number of the diagram. In both cases we have a contradiction with $\mathcal{D}$ being a minimal crossing number diagram of $\mathcal{T}$. Consequently, the number of strings on the inner crossings of $\mathcal{D}$ is at most two.\\

Assume now that there is only one string on the inner crossings. Then exactly two segments adjacent to the inner crossings are internal. Suppose these segments are adjacent to the same inner crossing, say $c_a$. It follows they have to be consecutive, otherwise more than one string would be on the inner crossings or we would have a contradiction to Lemma \ref{north west}. Then $c_b$ cannot be adjacent to three or four lateral segments, because in this case we would be in contradiction to Lemma \ref{north west}. Hence, two consecutive segments adjacent to $c_b$ are lateral and the other two segments are adjacent to the two inner crossings. In Fig. \ref{figure:inner crossings connected} we have an illustration of the projection of this case.\\

\begin{figure}[ht]
\begin{center}

\psset{unit=6mm,algebraic=true}
\begin{pspicture}(-5,-3)(5,3)

\pscircle[linewidth=0.1pt](0,0){3}

\pscurve(-1,0)(-0.5,0.1)(1.2,1)(1.4,0.4)

\pscurve(1,0)(0.5,-0.1)(-1.2,-1)(-1.4,-0.4)

\pscurve(-1.4,0.4)(2.5;110)(2.5;90)(2.5;60)(2.5;30)(2.5;0)(2.5;-30)(2.5;-60)(2.5;-90)(2.5;240)(2.5;210)(-2.3,-0.2)(-1.8,0)

\psline(3,0)(1,0)  \psline(1.4,-2.6533)(1.4,0.4) \psline(-1.4,-0.4)(-1.4,0.4)
\psline(-1.8,0)(-1,0)

\psarc[linestyle=dashed,dash=2pt 1.5pt](0,-3){1}{10}{170}
\rput(-1.7,0.3){$c_a$}
\rput(1.7,0.3){$c_b$}

\end{pspicture}
\end{center}

\caption{Only one string on the inner crossings with the two internal segments on the same inner crossing.}
\label{figure:inner crossings connected}
\end{figure}
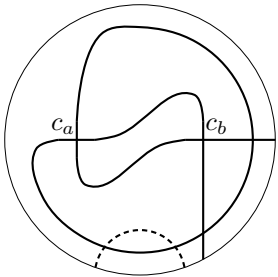
\noindent Suppose now each inner crossing is adjacent to exactly one internal segment. The ends of the inner crossings opposite to the internal segments cannot be adjacent to the other inner crossings or to a lateral segment in the direction of the outermost crossing adjacent to the internal segment it opposes, otherwise we would have a contradiction to only one string being on the inner crossings or to Lemma \ref{north west}. Therefore, the ends of the inner crossings opposite to the internal segments are connected to lateral segments in the direction of the outermost crossing adjacent to the other internal segment. The other segments adjacent to the inner crossings and outermost crossings are uniquely determined at this point. For this case we obtain the diagram shadow as in Fig. \ref{figure:inner crossings connected-case two}.\\

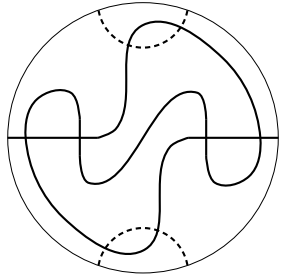
\begin{figure}[ht]
\begin{center}
\psset{unit=6mm,algebraic=true}
\begin{pspicture}(-5,-3)(5,3)

\pscircle[linewidth=0.1pt](0,0){3}

\pscurve(-1,0)(-0.6,0.2)(2.5;90)(2.5;45)(2.6;0)(2.5;-20)(1.6,-1)(1.4,-0.4)

\pscurve(1,0)(0.6,-0.2)(2.5;-90)(2.5;225)(-2.6;0)(2.5;160)(-1.6,1)(-1.4,0.4)

\pscurve(-1.4,-0.4)(-1.2,-1)(1.2,1)(1.4,0.4)

\psline(3,0)(1,0) \psline(-3,0)(-1,0) \psline(1.4,-0.4)(1.4,0.4) \psline(-1.4,-0.4)(-1.4,0.4)

\psarc[linestyle=dashed,dash=2pt 1.5pt](0,-3){1}{10}{170}
\psarc[linestyle=dashed,dash=2pt 1.5pt](0,3){1}{-170}{-10}

\end{pspicture}
\end{center}

\caption{Only one string on the inner crossings with the two internal segments on distinct inner crossings.}
\label{figure:inner crossings connected-case two}
\end{figure}

\noindent In this section all strings of the tangles are unknotted. Consequently, in this two cases, illustrated in Figs. \ref{figure:inner crossings connected} and \ref{figure:inner crossings connected-case two}, the string on the inner crossings is also unknotted. This necessarily means that two consecutive self-crossings are not alternating. Considering all consecutive self-crossings possibilities for the two cases we can obtain a diagram with a smaller crossing number than the one of $\mathcal{D}$, which contradicts $\mathcal{D}$ being a minimal crossing number diagram.
\end{proof}

\begin{lemma}\label{both inner}
If on an inner crossing there is only one string $s$, then $s$ is on the other inner crossing.
\end{lemma}
\begin{proof}
Suppose that on an inner crossing there is only one string $s$ and that $s$ is not on the other inner crossing. Then, the two internal segments of $s$ are adjacent to the same inner crossing, and the respective opposite segments are lateral. Hence, not to be in contradiction with Lemma \ref{north west}, we have a diagram projection of $\mathcal{D}$ as in Fig. \ref{figure:FigureA}, which is a contradiction to $\mathcal{T}$ being essential.\end{proof}

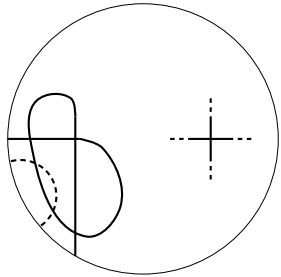
\begin{figure}[ht]
\begin{center}
\psset{unit=6mm,algebraic=true}
\begin{pspicture}(-5,-3)(5,3)

\pscircle[linewidth=0.1pt](0,0){3}

\pscurve(-1.5,0.5)(-1.6,0.9)(2.5;160)(2.5;180)(2.5;240)(-0.5,-1.5)(-0.9,-0.2)(-1.4,0)
\psline(3;-120)(-1.5,0.5)

\psline(1,0)(2,0) \psline(1.5,-0.5)(1.5,0.5) 
\psline[linestyle=dashed,dash=2pt 2pt](1.5,0.6)(1.5,0.9) \psline[linestyle=dashed,dash=2pt 2pt](1.5,-0.6)(1.5,-0.9) 
\psline[linestyle=dashed,dash=2pt 2pt](0.6,0)(0.9,0) \psline[linestyle=dashed,dash=2pt 2pt](2.1,0)(2.4,0) 

\psarc[linestyle=dashed,dash=2pt 1.5pt](3;205){0.8}{-55}{105}

\psline(-3,0)(-1.4,0)

 
\end{pspicture}
\end{center}

\caption{Only one string on an inner crossing and other strings on the other crossing.}
\label{figure:FigureA}
\end{figure}

{\noindent \bf Proof of Theorem \ref{tangle unknotted strings}.}
As before in this section, let $\mathcal{T}$ be a $n$-string essential tangle with all strings unknotted, and $\mathcal{D}$ a minimal diagram of $\mathcal{T}$. By the first part of the proof, we have $c(\mathcal{T})\geq 2n+2$. To prove the second part of the theorem suppose that $c(\mathcal{T})=2n+2$. As observed before, $\mathcal{D}$ has $2n$ outermost crossings and $2$ inner crossings.\\
From Lemma \ref{two strings} exactly two strings are on the inner crossings of $\mathcal{D}$. We denote the two strings on the inner crossings by $s_1$ and $s_2$. From Lemma \ref{both inner}, without loss of generality we can assume that $s_1$ is on both inner crossings.\\

(i) Suppose $s_2$ is only on the inner crossing $c_b$. Then in the projection corresponding to $\mathcal{D}$, $s_2$ contains two internal segments adjacent to $c_b$. One internal segment of $s_1$ is adjacent to $c_b$, and the opposite segment to this is adjacent to both $c_a$ and $c_b$, since $\mathcal{D}$ is minimal. One segment consecutive to the latter in $c_a$ is internal, and the remaining two segments adjacent to $c_a$ are lateral. (See Fig. \ref{figure:FigureB}.)\\

\begin{figure}[ht]
\begin{center}
\psset{unit=6mm,algebraic=true}
\begin{pspicture}(-5,-3)(5,3)

\pscircle[linewidth=0.1pt](0,0){3}

\psline(3;60)(3;-60)

\pscurve[linewidth=1.5pt](3;-120)(-1.5,0.5)(-1.5,0.5)(-1.4,1.5)(2.5;110)(2.5;90)(2.5;60)(2.5;30)(2.5;0)(2.5;-30)(2.5;-60)(2.5;-90)(2.5;-120)(2.5;-150)(-2.3,-0.2)(-1.8,0)(-1.8,0)(3,0)

\rput(-1.15,0.3){$c_a$}
\rput(1.8,0.3){$c_b$}
\rput(-1.8,1.2){$s_1$}
\rput(1.15,1.2){$s_2$}
\psarc[linestyle=dashed,dash=2pt 1.5pt](0,-3){1}{10}{170}
\end{pspicture}
\end{center}

\caption{Diagram when $s_2$ is only on the inner crossing $c_b$.}
\label{figure:FigureB}
\end{figure}
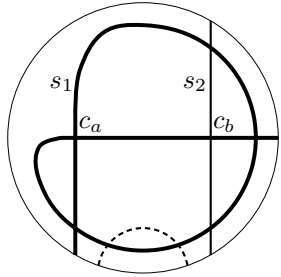

\noindent As $\mathcal{D}$ is a minimal diagram of $\mathcal{T}$, following the string $s_2$ the crossings in the diagram are alternating. As $s_1$ is unknotted, it is not alternating. Therefore it has two consecutive self-crossings of the same type, which are on opposite sides of $c_b$, because otherwise we can reduce the crossing number of $\mathcal{D}$. Then, from these two statements we can assume that $\mathcal{D}$ is alternating. There are two possible alternating diagrams but they correspond to homeomorphic tangles, which are depicted in Fig. \ref{figure:tangle unknotted strings}(a).\\

\noindent The tangle defined by $s_1$ and $s_2$ in $B$ is essential. In fact, we can add a $2$-string trivial tangle to the tangle with the strings $s_1$ and $s_2$ to obtain the alternating knot $6_2$ from Rolfsen's list in \cite{Rolfsen} that we denote by $K$. Suppose there is a disk separating $s_1$ and $s_2$. As $s_2$ is unknotted we could isotope it to the complement of $B$ in $S^3$, and realize that the knot $K$ is the connected sum of two unknots. This is a contradiction as $K$ is knotted. All other strings are parallel. Then if $\mathcal{T}$ is not essential there is a disk separating $s_1$ and $s_2$ from the other strings, and it suffices to study the case when the number of strings in $\mathcal{T}$ is three. Denote the third string by $s$ and suppose there is a disk in $B$ separating $s$ from $s_1\cup s_2$. Then, from the diagram $\mathcal{D}$, the string $s$ cobounds a disk $D_1$ with an arc in $\partial B$ intersecting $s_1$ once, and disjoint from $s_2$, and cobounds a disk $D_2$, with another arc in $\partial B$, disjoint from $s_1\cup s_2$. Therefore, the union of these two disks intersects $s_1\cup s_2$ in a point of $s_1$. As $s_1$ is not knotted, $D_1\cup D_2$ separates a ball from $B$ and an unknotted arc of $s_1$ from $s_1\cup s_2$ in it. Hence, this arc can be isotoped to an arc in $D_1$. With a proper choice of $D_1$, we can isotope $s_1$ in a way that reduces the crossing number of the tangle defined by $s_1$ and $s_2$ in $B$, which contradicts the crossing number minimality of the diagram with $s_1$ and $s_2$. Therefore, the tangle $\mathcal{T}$ is essential.\\

(ii) Suppose $s_1$ and $s_2$ are on both inner crossings. In this way, each inner crossing is adjacent to an internal segment for each string $s_1$ and $s_2$ and these segments are necessarily consecutive. For each inner crossing we can assume that the internal segments separate their opposite segments from the other inner crossings. In fact, if that is not the case, as in Fig. \ref{figure:FigureC}(a), then at least a segment opposite to an internal segment $l$ is lateral and in the direction of an outermost crossing adjacent to the same inner crossing $c$. In this case, we can isotope a string by taking the internal segment $l$ over (or under, depending on the outermost crossing type) the segment opposite to the other internal segment adjacent to $c$, without increasing the crossing number. 

\begin{figure}[ht]
\begin{center}
\psset{unit=6mm,algebraic=true}
$\begin{array}{cc}
\begin{pspicture}(-5,-4)(5,3)

\pscircle[linewidth=0.1pt](0,0){3}

\pscurve[linewidth=1.5pt](-1.5,0.5)(-1.6,0.9)(2.5;160)(2.5;180)(2.5;210)(2.5;240)(2.5;245)

\psline[linewidth=1.5pt](3;-120)(-1.5,0.5)

\psline(1,0)(2,0) \psline(1.5,-0.5)(1.5,0.5) 
\psline[linestyle=dashed,dash=2pt 2pt](-0.4,0)(-0.1,0)
\psline[linewidth=1.5pt,linestyle=dashed,dash=2pt 2pt](2.5;245)(2.5;260)
\rput(-1.7,-0.8){$l$}
\rput(-1.3,0.3){$c$}
\rput(-1,-1){$\nearrow$}

\psline(-3,0)(-0.5,0)

\rput(3.5;270){(a)} 
\end{pspicture}&
\begin{pspicture}(-5,-4)(5,3)

\pscircle[linewidth=0.1pt](0,0){3}

\pscurve[linewidth=1.5pt](-1.5,0.5)(-1.6,0.9)(2.5;160)(2.5;180)(2.5;210)(2.5;240)(2.5;245)
\psline[linewidth=1.5pt](-1.5,0.5)(-1.5,-0.5)
\pscurve[linewidth=1.5pt](-1.5,-0.5)(-1.4,-1.1)(-1,-1.5)(3;110)
\psline(-3,0)(0,0)
\psline[linestyle=dashed,dash=2pt 2pt](0.1,0)(0.4,0)
\psline[,linewidth=1.5pt,linestyle=dashed,dash=2pt 2pt](2.5;245)(2.5;260)


\psline(1,0)(2,0) \psline(1.5,-0.5)(1.5,0.5) 

\rput(-1.7,-0.8){$l$}
\rput(-1.3,0.3){$c$}

\rput(3.5;270){(b)} 
\end{pspicture}
\end{array}$
\end{center}

\caption{When $s_1$ and $s_2$ are on both inner crossings, we can assume that the inner crossings are as in (b).}
\label{figure:FigureC}
\end{figure}
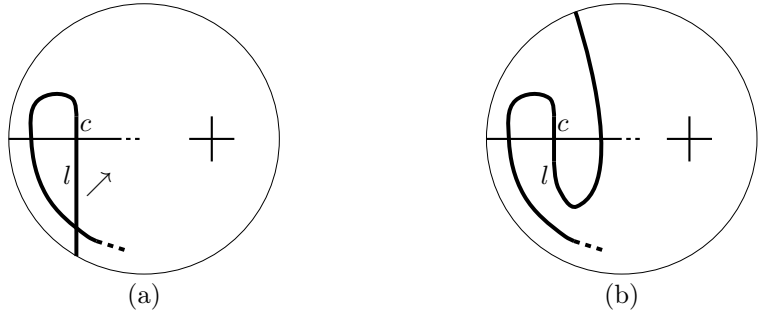

\noindent After this isotopy, $c$ is as claimed before: the internal segments of $c$ separate the corresponding opposite segments and the other inner crossing. (See Fig. \ref{figure:FigureC}(b).) Under this setting we are left with only one possible diagram projection of $s_1$ and $s_2$, as in Figure \ref{figure:FigureD}. There are two possible cases for the crossings of $\mathcal{D}$ in this case. Either the diagram is alternating, as in Fig. \ref{figure:tangle unknotted strings}(b), or the diagram is obtained from the alternating one by changing the crossings of a twist box, as in Fig. \ref{figure:tangle unknotted strings}(c).\\

\begin{figure}[ht]
\begin{center}
\psset{unit=6mm,algebraic=true}
\begin{pspicture}(-5,-3)(5,3)

\pscircle[linewidth=0.1pt](0,0){3}


\psline[linewidth=1.5pt](3;-120)(-1.5,0.5)
\pscurve[linewidth=1.5pt](-1.5,0.5)(-1.4,1.5)(2.5;110)(2.5;90)(2.5;70)(1.4,1.5)(1.5,0.5)
\psline[linewidth=1.5pt](3;-60)(1.5,0.5)

\psline(0.5,1)(0.5,2.95804)
\pscurve(1.2,0)(1,0.1)(0.6,0.5)(0.5,1)
\psline(1.8,0)(1.2,0)
\pscurve(1.8,0)(2.3,-0.2)(2.5;-30)(2.5;-60)(2.5;-90)(2.5;-120)(2.5;-150)(-2.3,-0.2)(-1.8,0)
\psline(-1.8,0)(-1.2,0)
\pscurve(-1.2,0)(-1,0.1)(-0.6,0.5)(-0.5,1)
\psline(-0.5,1)(-0.5,2.95804)

\psarc[linestyle=dashed,dash=2pt 1.5pt](0,-3){1}{10}{170} 
\psarc[linestyle=dashed,dash=2pt 1.5pt](3;50){1}{150}{310} 
\end{pspicture}
\end{center}

\caption{Diagram projection when $s_1$ and $s_2$ are on both inner crossings.}
\label{figure:FigureD}
\end{figure}
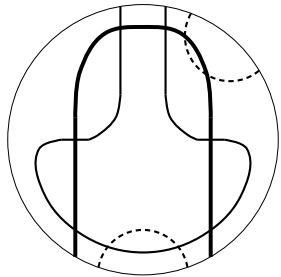

\noindent
To prove that these tangles are essential first note that we can add a $2$-string trivial tangle to the tangle with the strings $s_1$ and $s_2$ to get the square knot, in the case of Fig. \ref{figure:tangle unknotted strings}(b), or the granny knot, in the case of Fig. \ref{figure:tangle unknotted strings}(c). Then, if there is a properly embedded disk in $B$ separating the strings $s_1$, $s_2$ we would have a $2$-bridge decomposition of a composite knot, which is a contradiction because $2$-bridge knots are prime. To prove that the $n$-string tangle $\mathcal{T}$ is essential in these cases we now follow a similar argument as the one used in case (i).\\

The three tangles depicted in Fig. \ref{figure:tangle unknotted strings} are not equivalent. In fact, suppose there is a homeomorphism between two of these tangles, and consider the $2$-string trivial tangle added to the tangles in the argument and the respective knots obtained from adding this trivial tangle. Consider $B$ in $S^3$ and the disk of the diagram in $\mathbb{R}^2$. Note that the diagram of this trivial tangle is a collection of two disjoint arcs in the complement in $\mathbb{R}^2$ of the disk from the diagram of each essential tangle, with one arc at the top of this disk of the diagram and the other at the bottom. The homeomorphism between the essential tangles restricted to the boundary of the ball $B$ can be isotoped to either the identity or reflexion of $\partial B$, and from here extended to the ball in the complement of $B$ as the identity or reflection on the plane. Therefore, if two of these essential tangles are homeomorphic the knots obtained are also equivalent, which is a contradiction. Then, these three essential tangles are not equivalent. 
{\hfill \footnotesize $\square$}

\section{Crossing number of essential tangles with closed strings}

Let $\mathcal{T}$ be a $n$-string $k$-loop tangle with $k>0$, and $\mathcal{D}$ be a minimal diagram of $\mathcal{T}$. 

\begin{definition}
A {\em block} $\beta$ (with respect to $\mathcal{D}$) is a subtangle of $\mathcal{T}$ defined by a set of connected components of $\sigma$ such that:
\begin{enumerate}
\item if $l\in\beta$ is a loop, then every string or loop that crosses $l$ is in $\beta$;
\item if $s\in\beta$ is a string, then every loop that crosses $s$ is in $\beta$.
\end{enumerate}

\noindent The union of two blocks is a block and a block $\beta_1\subsetneq\beta_2$ is called a {\em subblock} of $\beta_2$. We call a crossing of a block $\beta_1$ with another block $\beta_2$ {\em outermost} (in $\beta_1$) if it separates the projection of $\beta_1$ into two components such that one of them has no crossings with other blocks.
\end{definition}

\noindent Notice that a string $s$ that doesn't cross any loop is a block, and a crossing of the string $s$ is outermost if and only if it is outermost as a crossing of the block $s$.

\begin{lemma}\label{crossingnumberminimaltangle}
If $\mathcal{T}$ has no subblocks, then $c(\mathcal{T})\geq 2(n+k-1)$.
\end{lemma}

\begin{proof}
Let $l_1$ be any loop of $\mathcal{T}$ and consider the subtangle $\mathcal{T}_1$ of $\mathcal{T}$ formed by $l_1$ and all the $n_1$ strings that cross $l_1$. Since every string of $\mathcal{T}_1$ crosses $l_1$ at least twice, then $c(\mathcal{T}_1)\geq 2 n_1$. If $\mathcal{T}_1\neq \mathcal{T}$, then, since $\mathcal{T}_1$ is not a subblock of $\mathcal{T}$, there is a loop $l_2$ in $\mathcal{T}-\mathcal{T}_1$ that crosses some component of $\mathcal{T}_1$, at least twice. Consider the subtangle $\mathcal{T}_2$ of $\mathcal{T}$ formed by $l_1$, $l_2$ and all the $n_2$ strings that cross $l_1$ or $l_2$. Again, since every string of $\mathcal{T}_2$ crosses $l_2$ at least twice, then $c(\mathcal{T}_2)\geq 2 (n_2+1)$. By repeating this argument, adding at stage $i$, one loop $l_i$ and the new strings of $\mathcal{T}$ that cross $l_i$, we obtain a subtangle $\mathcal{T}_i$ of $\mathcal{T}$ such that $c(\mathcal{T}_i)\geq 2 (n_i+i-1)$. At stage $k$ all the loops of $\mathcal{T}$ have been added, and, since $\mathcal{T}_k$ is not a subblock of $\mathcal{T}$, then $\mathcal{T}_k=\mathcal{T}$, which shows that
$
c(\mathcal{T})\geq 2 (n+k-1).
$

\end{proof}

We notice that this inequality is sharp, as the tangle $\mathcal{T}$ of the Fig. \ref{figure:tangle closed strings}, with one loop and $n$ strings, verifies the equality $c(\mathcal{T})=2n$. A similar argument to that on the end of the proof of Theorem 1 shows that $\mathcal{T}$ is essential. To see that this is the only $n$-string $1$-loop essential tangle with $c(\mathcal{T})=2n$, observe that every string must cross the loop twice so that there are no crossings between any two strings. 

\

Now we prove Theorem \ref{tangle closed strings}.\\




{\noindent \bf Proof of Theorem \ref{tangle closed strings}.}
Consider a partition of $\mathcal{T}$ into blocks. First note that if two blocks $\mathcal{T}_1$ and $\mathcal{T}_2$ of this partition cross at least twice, with $c(\mathcal{T}_i)\geq 2 (n_i+k_i-1)$, where $n_i$ and $k_i$ are respectively the number of strings and loops of $\mathcal{T}_i$, $i=1,2$, then $\mathcal{T}_1\cup \mathcal{T}_2$ is a subblock of $\mathcal{T}$ such that
$$
c(\mathcal{T}_1\cup \mathcal{T}_2)\geq 2(n_1+k_1-1)+2(n_2+k_2-1)+2=2\left((n_1+n_2)+(k_1+k_2)-1\vphantom{15mm}\right).
$$
We can therefore replace $\mathcal{T}_1$ and $\mathcal{T}_2$ by $\mathcal{T}_1\cup \mathcal{T}_2$ obtaining a smaller partition of $\mathcal{T}$ into blocks which have the property stated in Lemma \ref{crossingnumberminimaltangle}. 

\noindent Consider the collection of all minimal subblocks of $\mathcal{T}$ with respect to a minimal diagram of $\mathcal{T}$. Clearly this collection is a partition of $\mathcal{T}$ by blocks that verify the inequality in Lemma \ref{crossingnumberminimaltangle}. By repeating the above process on this collection, we eventually obtain a partition of $\mathcal{T}$ such that two blocks intersect each other at most once, and that also verifies the inequality in Lemma \ref{crossingnumberminimaltangle}. Next we prove that this partition has only one block. Being this the case, we have 
$$c(\mathcal{T})\geq 2(n+k-1),$$
\noindent as in the statement of the theorem.

\noindent In fact, suppose that the partition has more than one block, and that each block intersects each other at most once. Then all these blocks have at least two outermost crossings. 

\noindent Consider a block $\beta_1$ and an outermost crossing $c_1$ of $\beta_1$ (with $\beta_2$). Let $R_1$ be the outer region of the projection disk bounded by the component of $\beta_1-c_1$ that has no crossings with other blocks and that is the smallest that contains one of the two components of $\beta_2-c_1$. (See Fig. \ref{figure:Outerregion}.)

\begin{figure}[ht]
\begin{center}
\psset{unit=6mm,algebraic=true}
\begin{pspicture}(-5,-3)(5,3)
\begin{psclip}{\pscircle[linestyle=none](0,0){3}}
\pscurve[fillcolor=lightgray,fillstyle=solid,linestyle=none](4,-4)(1,-1)(0.5,0)(1,1)(1.5,1.5)(2,10)
\pscurve[fillcolor=lightgray,fillstyle=solid,linestyle=none](1,0.5)(0,2)(3;80)(10;60)
\pspolygon[fillcolor=white,fillstyle=solid,linestyle=none](0,0)(3,0)(3,-3)(0,-3)
\end{psclip}

\pscircle[linewidth=0.1pt](0,0){3}

\psline(0,0)(3,0)
\pscurve(1,-1)(0.5,0)(1,1)
\pscurve(1,0.5)(0,2)(3;80)
 
\rput(1.8,1.5){$R_1$}
\rput(0.3,-0.3){$c_1$}
\rput(2,-0.4){$\beta_1$}
\rput(2.3;100){$\beta_2$}

\end{pspicture}
\end{center}

\caption{The outer region defined by an outermost crossing.}
\label{figure:Outerregion}
\end{figure}
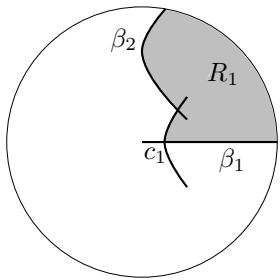

\noindent If $\beta_2$ has no outermost crossing in $R_1$, then we can reduce the number of crossings by an isotopy of $\beta_1$. Otherwise, consider an outermost crossing $c_2$ of $\beta_2$ (with $\beta_3$). Note that $\beta_3\neq\beta_1$ as we are assuming that two blocks do not share more than one crossing. We can similarly consider the outer region $R_2$ of the projection disk bounded by the component of $\beta_2-c_2$ that has no crossings with other blocks and that is the smallest that contains one of the two components of $\beta_3-c_2$. Notice that $R_2$ has fewer crossings than $R_1$. Repeating this process we get a contradiction to the diagram being  minimal or to the number of crossings being finite.
{\hfill \footnotesize $\square$}

\end{document}